\documentclass[fleqn,reqno,10pt,a4paper,final]{amsart}

\usepackage[a4paper,left=35mm,right=35mm,top=30mm,bottom=30mm,marginpar=25mm]{geometry} 
\usepackage{amsmath}
\usepackage{amssymb}
\usepackage{amsthm}
\usepackage{amscd}
\usepackage[ansinew]{inputenc}
\usepackage{cite}
\usepackage{bbm}
\usepackage{color}
\usepackage[english=american]{csquotes}
\usepackage[final]{graphicx}
\usepackage{hyperref}

\graphicspath{{Pictures/}}

\numberwithin{equation}{section}

\newtheoremstyle{thmlemcorr}{10pt}{10pt}{\itshape}{}{\bfseries}{.}{10pt}{{\thmname{#1}\thmnumber{ #2}\thmnote{ (#3)}}}
\newtheoremstyle{thmlemcorr*}{10pt}{10pt}{\itshape}{}{\bfseries}{.}\newline{{\thmname{#1}\thmnumber{ #2}\thmnote{ (#3)}}}
\newtheoremstyle{defi}{10pt}{10pt}{\itshape}{}{\bfseries}{.}{10pt}{{\thmname{#1}\thmnumber{ #2}\thmnote{ (#3)}}}
\newtheoremstyle{remexample}{10pt}{10pt}{}{}{\bfseries}{.}{10pt}{{\thmname{#1}\thmnumber{ #2}\thmnote{ (#3)}}}
\newtheoremstyle{ass}{10pt}{10pt}{}{}{\bfseries}{.}{10pt}{{\thmname{#1}\thmnumber{ A#2}\thmnote{ (#3)}}}

\theoremstyle{thmlemcorr}
\newtheorem{theorem}{Theorem}
\numberwithin{theorem}{section}

\newtheorem{corollary}[theorem]{Corollary}
\newtheorem{proposition}[theorem]{Proposition}

\theoremstyle{thmlemcorr*}
\newtheorem{theorem*}{Theorem}
\newtheorem{lemma*}[theorem]{Lemma}
\newtheorem{corollary*}[theorem]{Corollary}
\newtheorem{proposition*}[theorem]{Proposition}
\newtheorem{problem*}[theorem]{Problem}
\newtheorem{conjecture*}[theorem]{Conjecture}

\theoremstyle{defi}
\newtheorem{definition}[theorem]{Definition}

\theoremstyle{remexample}
\newtheorem{remark}[theorem]{Remark}

\theoremstyle{ass}
\newtheorem{assumption}{Assumption}

\newcommand{\Bcal}{\mathcal{B}}
\newcommand{\Ccal}{\mathcal{C}}

\newcommand{\Fcal}{\mathcal{F}}

\newcommand{\Ucal}{\mathcal{U}}

\newcommand{\Xcal}{\mathcal{X}}
\newcommand{\Ycal}{\mathcal{Y}}

\newcommand{\N}{\mathbb{N}}
\newcommand{\R}{\mathbb{R}}
\newcommand{\C}{\mathbb{C}}

\newcommand{\Z}{\mathbb{Z}}

\def\XXint#1#2#3{{\setbox0=\hbox{$#1{#2#3}{\int}$} 
\vcenter{\hbox{$#2#3$}}\kern-.5\wd0}}

\newcommand{\be}{\begin{equation}}
\newcommand{\en}{\end{equation}}
\newcommand{\p}{\partial}
\renewcommand{\l}{{\lambda}}

\renewcommand{\phi}{\varphi}

\title{On the spectrum of shear flows and uniform ergodic theorems}

\author{Jonathan Ben-Artzi}

\address{Department of Applied Mathematics and Theoretical Physics, University of Cambridge, Centre for Mathematical Sciences, Wilberforce Road, Cambridge CB3 0WA, United Kingdom.}
\email{J.Ben-Artzi@damtp.cam.ac.uk}

\keywords{First Order Operator, Density of States, Weighted Spaces, Uniform Ergodic Theorem, Singular Continuous Spectrum, Shear Flow}

\subjclass[2010]{Primary 34L15, Secondary 35P15, 47A35, 37A30}

\begin{document}
\maketitle

\begin{abstract}
The spectra of parallel flows (that is, flows governed by first-order differential operators parallel to one direction) are investigated, on both $L^2$ spaces and weighted-$L^2$ spaces. As a consequence, an example of a flow admitting a purely singular continuous spectrum is provided. For flows admitting more regular spectra the density of states is analyzed, and spaces on which it is uniformly bounded are identified. As an application, an ergodic theorem with uniform convergence is proved. 

\vspace{4pt}

%

\vspace{4pt}

\noindent\textsc{Date:} \today.
\end{abstract}

\tableofcontents

\section{Introduction}\label{sec:intro}
We study operators of the form
	\begin{equation}\label{eq:shear-flow}
	H_1=-i\psi\frac{\partial}{\partial x_1},\quad\text{ on }L^2(\R^d)
	\end{equation}
and
	\begin{equation}\label{shear-flow-weighted}
	H_w=-i\frac{\psi}{w}\frac{\partial}{\partial x_1},\quad\text{ on }L^2_w(\R^d)
	\end{equation}
where $0<\psi\in L^\infty(\R^d)$ depends only on $x'=(x_2,\dots,x_d)$, $0<w\in L^\infty(\R^d)$ is a weight function and $L^2_w$ the weighted-$L^2$ space endowed with the inner product
	\begin{equation}
	(f,g)_{L^2_w}=\int_{\R^d}f\overline{g}w.
	\end{equation}
We naturally call such operators \emph{shear flows} which is the standard term in fluid dynamics for flows that have straight and parallel flow lines.  Since $\psi$ and $w$ are assumed to be real-valued both $H_1$ and $H_w$ are symmetric. Self-adjointness of $H_1$ (under mild conditions on $\psi$) is standard (Corollary \ref{cor:derivative}), but that is not the case for $H_w$. In Theorem \ref{thm:self-adjoint-weighted} we give sufficient conditions on $w$ and find an appropriate domain so that $H_w$ is self-adjoint. We characterize the spectrum of $H_w$ and give an explicit example where the spectrum is purely singular continuous.

In cases where the spectrum is more well-behaved, we study the density of states of both operators and identify spaces $\Xcal^\sigma\subset L^2(\R^d)$ and $\Ycal^\sigma\subset L^2_w(\R^d)$ on which there is an explicit estimate for the density of states of both operators (the parameter $\sigma\in\R$ is related to the behavior at infinity). Letting $G_t=e^{itH_w}$ be the one-parameter unitary group of transformations generated by $H_w$ and letting $P$ be the orthogonal projection onto $\{f\in L^2_w(\R^d)\ | \ f\circ G_t=f\}$, we use the estimate on the density of states to obtain the \emph{uniform} convergence
	\begin{equation}
	\lim_{T\to\infty}\frac{1}{2T}\int_{-T}^{T}G_t\ dt=P\quad\text{ in }\Bcal(\Ycal^\sigma,\Ycal^{-\sigma})
	\end{equation}
where $\Ycal^{-\sigma}$ is the dual space to $\Ycal^\sigma$ with respect to the inner product in $L^2_w$ and $\Bcal(\Ycal^\sigma,\Ycal^{-\sigma})$ denotes the space of bounded linear operators from $\Ycal^\sigma$ to $\Ycal^{-\sigma}$. The proof follows the ideas of von Neumann \cite{VonNeumann1932a} in his proof of the ergodic theorem. It diverges from von Neumann's proof in that we replace the Stieltjes measure $d(E(\lambda)f,g)$ by its density using the estimates on the density of states.

The existing literature on spectra of first-order differential operators, it appears, has primarily been in relation to Euler's equations for incompressible fluids, in particular in two-dimensions and in bounded domains. Recently, Cox \cite{Cox2013a} studied the spectrum of the linearization of Euler's equations in the vorticity formulation. Specifically, he focused on the spectrum due to \emph{periodic} trajectories of the flow. These results are analogous to our results for \emph{unbounded} trajectories in \emph{weighted} spaces, see Corollary \ref{cor:spectrum-confined}. We refer to the references within \cite{Cox2013a} as well as the survey article \cite{Shvydkoy2005} for further discussion.

The literature on ergodic theory is vast, going back to von Neumann \cite{VonNeumann1932a} and Birkhoff \cite{Birkhoff1931}. In his proof, von Neumann for the first time provided a concrete example of the benefits of the spectral theorem for operators in Hilbert spaces, and, in particular, the resolution of the identity related to self-adjoint operators. Referring to \cite{Koopman1931}, he says:
	\begin{quote}
	\emph{The pith of the idea in Koopman's method resides in the conception of the spectrum $E(\lambda)$ reflecting, in its structure, the properties of the dynamical system -- more precisely, those properties of the system which are true ``almost everywhere,'' in the sense of Lebesgue sets.}
	\end{quote}
More recent results often treat discrete ergodic theorems, e.g. in relation to the theory of Ces\`aro sums. For instance, in \cite{Kakutani1981} it shown that such sums may converge very slowly. The relationship between the behavior of the spectral measure and the density of states -- particularly near zero -- and ergodic averages has been the subject of study of some in the Russian school. See \cite{Dzhulai2011} for instance, and the references therein.

In a separate paper \cite{Ben-Artzi2012b} together with a coauthor, the problem of extending the present results to more general first-order differential operators (that is, vectorfields) in $\R^d$ and even on Riemannian manifolds, is being addressed. The main hurdle is in finding a satisfactory set of sufficient conditions for such vectorfields to be rectifiable and hence adhere to the results presented here.

\section{Background}\label{sec:back}

\subsection{Self-adjoint operators}
The spectral theorem is the central tool we use in this work. As a fundamental theorem in modern analysis, there is no need to state it here. However, as it is only applicable to self-adjoint operators, it is essential to study the self-adjointness of differential operators which are formally symmetric. This becomes especially delicate when considering operators on weighted spaces, see Section \ref{sec:trans-op-weighted}. In this section we use the letter $\mathcal H$ to denote a Hilbert space, and $H$ to denote a generic linear operator on it. The domain of $H$ is denoted $D(H)$. We shall always specify what precise assumptions are imposed on them.

\begin{remark}
When there is no risk for confusion we shall abuse notation and use the same symbol for an essentially self-adjoint operator and its unique self-adjoint extension.
\end{remark}

For the following definition we follow \cite[IV-\S5.1 and V-\S3.4]{Kato1995}:
\begin{definition}
The \emph{deficiency} of $H:D(H)\subset\mathcal H\to\mathcal H$ is the codimension of the range of $H$ in $\mathcal H$:
	$$
	\operatorname{def}(H)=\dim\left(\mathcal H/\operatorname{Ran}(H)\right).
	$$
If $H$ is closed and symmetric then we define its \emph{deficiency index} to be
	$$
	(m^-,m^+):=(\operatorname{def}(H-i),\operatorname{def}(H+i)).\footnote{It is well known that the deficiency of a symmetric operator is constant on $\C^\pm$, see \cite[V-\S3.4]{Kato1995}.}
	$$
\end{definition}

The following result is very useful for determining (essential) self-adjointness of operators. For proofs and discussions see \cite[Theorem VIII.3]{Reed1981}, \cite[V, Theorem 3.16]{Kato1995} or \cite{Tao2011b}.

\begin{proposition}[Basic criterion for self-adjointness]\label{prop:self-adjoint-condition}
Let $H:D(H)\subset\mathcal H\to\mathcal H$ be densely defined and symmetric. Then $H$ is essentially self-adjoint if and only if $\operatorname{Ran}(H\pm i)$ are both dense in $\mathcal H$. If $H$ is also closed, then it is self-adjoint if and only if $\operatorname{Ran}(H\pm i)=\mathcal H$ (that is, the deficiency index of $H$ is $(0,0)$).
\end{proposition}

We refer to T. Tao's blog \cite{Tao2011b} for a discussion of the following proposition, which is essentially a reformulation of Stone's theorem:

\begin{proposition}\label{prop:schrodinger-criterion}
Let $H:D(H)\subset\mathcal H\to\mathcal H$ be densely defined and symmetric and suppose that for every $f\in D(H)$ there exists a continuously (strongly) differentiable function $u:\R\to D(H)$ solving
	$$
	\frac{\p u}{\p t}+iHu=0, \quad
	u(0)=f.
	$$
Then $H$ is essentially self-adjoint.
\end{proposition}

As expected, transport operators corresponding to divergence-free vectorfields are essentially self-adjoint. Since this is not trivial, we quote the following result which appears in \cite{Tao2011b} and whose proof relies on Proposition \ref{prop:schrodinger-criterion}.

\begin{proposition}\label{thm:transport-self-adjoint}
Let $(\mathcal M,g)$ be a Riemannian manifold and let $u$ be a divergence-free vectorfield on $\mathcal M$ (with respect to the metric $g$). Let $X(s,x)$ denote the integral curves (trajectories) of the flow, solving the equation $\frac{d}{ds}X(s,x)=u(X(s,x))$ with initial conditions $X(0,x)=x$, and assume that for each $x\in\mathcal M$ the solution of this equation exists for all $s$. Then the operator
	$$
	H=-iu\cdot\nabla:C_0^\infty(\mathcal M)\subset L^2(\mathcal M)\to L^2(\mathcal M)
	$$
is essentially self-adjoint.
\end{proposition}

\begin{corollary}\label{cor:derivative}
The operator $T=-i\frac{d}{dx}:C_0^\infty(\R)\subset L^2(\R)\to L^2(\R)$ and, more generally, the shear flow operator $H_1=-i\frac{\p}{\p x_1}:C_0^\infty(\R^d)\subset L^2(\R^d)\to L^2(\R^d)$ are both essentially self-adjoint.
\end{corollary}

\subsection{The spectral family}\label{sec:spec-fam}
Consider some self-adjoint operator $H:D(H)\subset\mathcal{H}\to\mathcal{H}$. Recall first the characterization of its spectral family (also known as a \emph{resolution of the identity}). The spectral family $\{E(\lambda)\}_{\lambda\in\R}$ of $H$ is a family of projection operators in $\mathcal{H}$ with the property that, for each $\lambda\in\R$, the subspace $\mathcal{H}^\lambda=E(\lambda)\mathcal{H}$ is the largest closed subspace such that
\begin{enumerate}
\item $\mathcal{H}^\lambda$ \emph{reduces} $H$, namely, $HE(\lambda)g=E(\lambda)Hg$ for every $g\in D(H)$. In particular, if $g\in D(H)$ then also $E(\lambda)g\in D(H)$.
\item $(Hu,u)_\mathcal H\leq \lambda (u,u)_\mathcal H$ for every $u\in\mathcal{H}^\lambda\cap D(H)$.
\end{enumerate}

\subsubsection{The spectral measure and its absolutely continuous part}
Given any $f,g\in\mathcal H$ the spectral family defines a complex function of bounded variation on the real line, given by
	\be\label{eq:spec-meas}
	\R\ni\lambda\mapsto(E(\lambda)f,g)_{\mathcal H}.
	\en
It is well-known that such a function gives rise to a measure (depending on $f,g$) called the \emph{spectral measure}. Recall the following useful fact:
\begin{proposition}[{\cite[X-\S1.2, Theorem 1.5]{Kato1995}}]\label{prop:abs-cont-subs}
Let $U\subset\R$ be open. The set of $f,g\in\mathcal H$ for which the spectral measure is absolutely continuous in $U$ with respect to the Lebesgue measure forms a closed subspace $\mathcal A\subset\mathcal H$. This subspace is referred to as the \emph{absolutely continuous subspace of $H$ on $U$}.
\end{proposition}

\subsubsection{The density of states}\label{sec:dens-states}
Let $\mathcal{A}\subset\mathcal H$ be the absolutely continuous subspace of $H$ on $U$ and let $\lambda_0\in U$. If there exists a subspace $\mathcal X\subset\mathcal{A}$ equipped with a stronger norm such that the bilinear form $\frac{d}{d\lambda}\big|_{\lambda=\lambda_0}(E(\lambda)\cdot,\cdot)_{\mathcal H}:\mathcal X \times \mathcal X\to\C$ is bounded then it induces a bounded operator $A(\lambda_0):\mathcal X\to\mathcal X^*$ defined via
	\be\label{eq:A-lambda}
	\left<A(\lambda_0)f,g\right>
	=
	\frac{d}{d\lambda}\Big|_{\lambda=\lambda_0}(E(\lambda)f,g)_{\mathcal H},\quad f,g\in\Xcal,
	\en
where $\left<\cdot,\cdot\right>$ is the $(\mathcal X^*,\mathcal X)$ dual-space pairing, and $\mathcal X^*$ is the dual of $\mathcal X$ with respect to the inner-product on $\mathcal H$. We refer to both the bilinear form $\frac{d}{d\lambda}\big|_{\lambda=\lambda_0}(E(\lambda)\cdot,\cdot)_{\mathcal H}$ and the operator $A(\lambda_0)$ as the \emph{density of states} of the operator $H$ at $\lambda_0$. In physics, the density of states at $\lambda_0$ represents the number possible states a system can attain at the energy level $\lambda_0$.

\subsection{Multiplication operators}
Multiplication operators have the added benefit of having a spectral family that is completely identifiable. Let $\Omega\subset\R^d$ be a domain that is either bounded or unbounded and consider the self-adjoint operator $H:D(H)\subset L^2(\Omega)\to L^2(\Omega)$ mapping
	$$
	H: u\mapsto mu
	$$
where $m:\Omega\to\R$ is a real locally bounded function. Then we can make the following relevant observation \cite[X-\S1.2, Example 1.9]{Kato1995} :

\begin{proposition}\label{cl:spec-meas}
The projection $E(\lambda)$ where $\lambda\in\R,$ is the set of $u\in L^2(\Omega)$ such that the set $\{x\in\Omega\ |\ u(x)\neq0,\ m(x)>\lambda\}$ has zero measure.
\end{proposition}

\begin{proof}
Any $u\in L^2(\Omega)$ that vanishes on the set $\{m>\lambda\}$ is clearly an element of $E(\lambda)L^2(\Omega)$ by the definition of $E(\lambda)L^2(\Omega)$.
Conversely suppose that there is some $v\in E(\lambda) L^2(\Omega)\cap D(H)$ such that the set $A=\{x\in \Omega\ |\ v(x)\neq0,\ m(x)>\lambda\}$ has positive measure. Clearly this set cannot be the full support of $v,$ since in such a case $(Hv,v)_{L^2(\Omega)}>\lambda(v,v)_{L^2(\Omega)}.$ Thus, the set $B=\{x\in\Omega\ |\ v(x)\neq0,\ m(x)\leq\lambda\}$ has also positive measure. By the beginning of the proof, the restriction  $w=\chi_Bv$ of $v$ to $B$ (zero outside) is certainly in $E(\lambda) L^2(\Omega).$ It follows that also $v-w$ is in the subspace, and is supported in $A,$ where $v(x)-w(x)=v(x)\neq0$ and $m(x)>\l$.  But then $u=v-w$ satisfies $(Hu,u)_{L^2(\Omega)}>\lambda (u,u)_{L^2(\Omega)}$, in contradiction to the  definition of the spectral projection.
\end{proof}

\subsection{Direct integrals}\label{sec:dir-int}
Shear flows (and all regular flows, for that matter) admit a natural decomposition into fibers: each flow line is independent of any other flow line, and on it the shear flow operator reduces to a one-dimensional derivative. Hence, considering the (more general) weighted case, the space
	\be\label{def:weighted-space}
	L^2_w(\R^d)
	=
	\left\{
	f:\R^d\to\C\ \Big|\ \|f\|_w^2:=\int_{\R^d}|f(x)|^2w(x)\ dx<\infty
	\right\}
	\en
is decomposable as
	\be
	L_w^2(\R^d)
	=
	\int_{\R^{d-1}}^\oplus L_{w(\cdot,x')}^{2}(\R)\ dx',
	\en
where 
	\be
	 L_{w(\cdot,x')}^{2}(\R)
	=
	\left\{
	g:\R\to\C\ \Big|\ \|g\|_{w(\cdot,x')}^2:=\int_\R|g(s)|^2w(s,x')\ ds<\infty
	\right\}.
	\en
The shear flow operator $H_w:L^2_w(\R^d)\to L^2_w(\R^d)$ decomposes accordingly as
	\begin{equation}\label{eq:direct-int-decomposition}
	H_w=\int_{\R^{d-1}}^\oplus H_w^{x'}\ dx'
	\end{equation}
where
	\begin{equation}
	H_w^{x'}=-i\frac{\psi(x')}{w(\cdot,x')}\frac{d}{dx},\quad x'\in\R^{d-1}
	\end{equation}
is an operator $L_{w(\cdot,x')}^{2}(\R)\to L_{w(\cdot,x')}^{2}(\R)$. In Theorem \ref{thm:self-adjoint-weighted} we carefully state what are the domains of these operators. Naturally, this decomposition induces a relationship between the spectrum of each fiber $H_w^{x'}$ and the spectrum of $H_w$. Most importantly, the spectral family is related through
	\begin{equation}\label{eq:direct-int-spec-family}
	E_w(\lambda)
	=
	\int_{\R^{d-1}}^\oplus E^{x'}_w(\lambda)\ dx'
	\end{equation}
where $\{E^{x'}_w(\lambda)\}_{\lambda\in\R}$ is the spectral family of the operator $H_w^{x'}$ and the spectrum itself satisfies
	\be\label{eq:direct-int-eigenvalues}
	\lambda\in\Sigma(H_w)\quad\Leftrightarrow\quad\forall\epsilon>0,\ \mathcal{L}^{d-1}\left\{x'\ \big|\ \Sigma(H_w^{x'})\cap(\lambda-\epsilon,\lambda+\epsilon)\neq\emptyset\right\}>0
	\en
where $\mathcal L^{d-1}$ denotes the $d-1$ dimensional Lebesgue measure. These facts will be useful later. Another important observation, a corollary of equations \eqref{eq:direct-int-decomposition} and \eqref{eq:direct-int-eigenvalues}, is that if $S\subset\R^{d-1}$ is measurable and $S^c:=\R^{d-1}\setminus S$ is its complement, then
	\begin{equation}\label{eq:decomp-s}
	H_w=H_w^S\oplus H_w^{S^c}:=\int_{S}^\oplus H_w^{x'}\ dx'\oplus \int_{S^c}^\oplus H_w^{x'}\ dx',
	\end{equation}
where the measures on $S$ and its complement are the ones induced from the restriction of the Lebesgue measure on $\R^{d-1}$. Hence we can conclude in addition that
	\begin{equation}\label{eq:decomp-of-spectrum}
		\Sigma(H_w)=\Sigma(H_w^S)\cup\Sigma(H_w^{S^c}).
	\end{equation}
The notion of a direct integral is originally due to von Neumann \cite{VonNeumann1949}, who called it a  \emph{generalized direct sum}. We refer to \cite[\S 16]{Reed1978} or \cite{Mautner1948} as additional references.

\section{The spectrum in weighted-$L^2$ spaces}\label{sec:trans-op-weighted}
In this section we establish properties of transport operators in weighted spaces, and, in particular, adapt Corollary \ref{cor:derivative} to this situation. For simplicity, we first treat the simpler one-dimensional case, and then use our observations to prove Theorem \ref{thm:self-adjoint-weighted} for parallel flows in higher dimensions.

\subsection{The one-dimensional case}
Let $0<w\in L^1(\R)\cap L^\infty(\R)$ be a real-valued positive weight function, and let $L^2_w(\R)\supset L^2(\R)$ be the weighted space defined in \eqref{def:weighted-space} with $d=1$.
Given a self-adjoint operator $T_0:D(T_0)\subset L^2(\R)\to L^2(\R)$, it is natural to expect that the operator $T_w=w^{-1}T_0$ acting in $L^2_w(\R)$ with domain included in $D(T_0)$ is self-adjoint as well, or is at least essentially self-adjoint. As we show below this is not true in general. Though this is discouraging, we show that in the case of interest $T_0=T=-i\frac{d}{dx}$ there exists an appropriate choice of domain that allows for the extraction of a particular self-adjoint extension of $T_w$.

As we have already seen in Corollary \ref{cor:derivative} the operator $T=-i\frac{d}{dx}:C_0^\infty(\R)\subset L^2(\R)\to L^2(\R)$ is essentially self-adjoint. It is well known (basically by definition) that its unique self-adjoint extension, which we continue to denote by $T$, has  $H^1(\R)$ as its domain, where $H^1(\R)$ is the usual Sobolev space consisting of functions $f$ such that $f,f'\in L^2(\R)$, $f'$ being the distributional derivative.

\begin{theorem}\label{thm:weighted-1d-sa}
Consider the self-adjoint operator $T=-i\frac{d}{dx}:H^1(\R)\subset L^2(\R)\to L^2(\R)$. Define the sets
	$$
	D_0=\left\{f\in H^1(\R)\ \big|\ w^{-\frac12}Tf\in L^2(\R)\right\}
	$$
and
	$$
	D_1^\alpha=\left\{f\in L^2_w(\R)\ \Big| \ T_wf\in L^2_w(\R),\ \lim_{x\to\infty}f(x)=\alpha\lim_{x\to-\infty}f(x)\right\},\quad|\alpha|=1.
	$$
Then the symmetric and densely defined operator
	\be\label{eq:weighted-operator}
	T_w=w^{-1}T:D_0\subset L^2_w(\R)\to L^2_w(\R)
	\en
is \emph{not} essentially self-adjoint, but each of the one-parameter family of operators
	\be\label{eq:weighted-operator}
	T_w^\alpha=w^{-1}T:D_1^\alpha\subset L^2_w(\R)\to L^2_w(\R),\quad|\alpha|=1,
	\en
\emph{is} essentially self-adjoint. We denote by $H_w^\alpha$ the unique self-adjoint extension. The spectrum of $H_w^\alpha$ consists only of eigenvalues. If $\alpha=e^{i\beta},\ \beta\in[0,2\pi)$, then the eigenvalues are given by
	$$
	\lambda_k^\beta
	=
	\|w\|_{L^1(\R)}^{-1}(\beta+2\pi k),\quad k\in\Z.
	$$ 

\end{theorem}

\begin{proof}
It is clear that $T_w$ is symmetric, closed and densely defined on $D_0$. Before showing that it  is not essentially self-adjoint on the domain $D_0$, we remark that $D_0$ is natural to consider: indeed, the condition $w^{-\frac12}Tf\in L^2(\R)$ is equivalent to $T_wf=w^{-1}Tf\in L^2_w(\R)$. Moreover, for such functions $f$ the limits at $\pm \infty$ always exist (and are zero) since $\frac{df}{dx}\in L^1(\R)$ by an elementary Cauchy-Schwarz inequality using the fact that $w\in L^1(\R)$: $\|\frac{df}{dx}\|_{L^1(\R)}\leq\|w^{-\frac12}\frac{df}{dx}\|_{L^2(\R)}\|w\|_{L^1(\R)}<\infty$.

To show that $T_w$ is not essentially self-adjoint on the domain $D_0$ we show that the containment $T_w\subset T_w^*$ is \emph{proper}. Therefore our starting point is identifying the adjoint operator, which, in our case, is actually the maximal operator associated with the differential operator $-iw^{-1}\frac{d}{dx}$ acting in $L^2_w(\R)$. To do this, given $g\in L^2_w(\R)$ we look for $h\in L^2_w(\R)$ such that
	$$
	(T_wf,g)_{L^2_w(\R)}=(f,h)_{L^2_w(\R)},\quad\forall f\in D_0.
	$$
In particular, it holds that for any $f\in C_0^\infty(\R)$
	$$
	-i\int_\R\frac{d}{dx}f(x)\overline{g(x)}\ dx=\int_\R f(x)\overline{h(x)}w(x)\ dx
	$$
and we get (in the sense of distributions)
	$$
	\frac{d}{dx}g(x)=-i{h(x)}w(x).
	$$
Therefore
	$$
	\int_\R w(x)^{-1}\left|\frac{d}{dx}g(x)\right|^2\ dx=\int_\R|h(x)|^2w(x)\ dx<\infty.
	$$
We conclude that if the domain of $T_w$ is taken to be $D_0$ then the adjoint has domain
	$$
	D(T_w^*)=\left\{g\in L^2_w(\R) \ \Big| \ -iw^{-1}\frac{d}{dx}g\in L^2_w(\R)\right\}.
	$$
To show that $T_w$ is not essentially self-adjoint on $D_0$ we need to solve $(T_w^*\pm i)g=0$ (see Proposition \ref{prop:self-adjoint-condition}). Such solutions are given by
	$$
	g_\pm(x)=Ce^{\pm\int_0^xw(t) dt}.
	$$
Clearly $\pm w^{-1}g'_\pm=g_\pm\in L^2_w(\R)$ so that $g\in D(T_w^*)$. However $g\notin D_0$ since it does not decay. In fact, we have just shown that the deficiency index of $T_w$ with domain $D_0$ is $(1,1)$.

To show that $T_w$ is essentially self-adjoint on the domain $D_1^\alpha$ we let $g\in L^2_w(\R)$ and seek $h\in L^2_w(\R)$ such that
	\be\label{eq:int-by-parts2}
	(T_wf,g)_{L^2_w(\R)}=(f,h)_{L^2_w(\R)},\quad\forall f\in D_1^\alpha.
	\en
Note that $D_1^\alpha\supset D_0$ for every $\alpha$. As before, by taking $f$ to be a smooth, compactly supported test function we can conclude that $g$ is differentiable and $-iw^{-1}\frac{d}{dx}g\in L^2_w(\R)$. However $C_0^\infty(\R)$ is not a core. Let $f\in C^\infty(\R)$ be such that $\lim_{x\to\infty}f(x)=\alpha\lim_{x\to-\infty}f(x)$. The left hand side of \eqref{eq:int-by-parts2} becomes
	\begin{align*}
	(T_wf,g)_{L^2_w(\R)}
	&=
	-i\int_{-\infty}^\infty\frac{d}{dx}f(x)\overline{g(x)}\ dx\\
	&=
	-i\lim_{R\to\infty}\int_{-R}^R\frac{d}{dx}f(x)\overline{g(x)}\ dx\\
	&=
	i\lim_{R\to\infty}\left[\int_{-R}^Rf(x)\frac{d}{dx}\overline{g(x)}\ dx-f(R)\overline{g(R)}+f(-R)\overline{g(-R)}\right].
	\end{align*}
Since $-iw^{-1}\frac{d}{dx}g\in L^2_w(\R)$ all limits exist so that we obtain
	\begin{align*}
	(T_wf,g)_{L^2_w(\R)}
	&=
	i\left[\int_{\R}f(x)\frac{d}{dx}\overline{g(x)}\ dx-f(\infty)\overline{g(\infty)}+f(-\infty)\overline{g(-\infty)}\right]\\
	&=
	i\left[\int_{\R}f(x)\frac{d}{dx}\overline{g(x)}\ dx-\alpha f(-\infty)\overline{g(\infty)}+f(-\infty)\overline{g(-\infty)}\right]\\
	&=
	i\int_{\R}f(x)\frac{d}{dx}\overline{g(x)}\ dx-if(-\infty)\left(\alpha\overline{g(\infty)}-\overline{g(-\infty)}\right)
	\end{align*}
which must equal the right hand side of \eqref{eq:int-by-parts2}:
	$$
	i\int_{\R}f(x)\frac{d}{dx}\overline{g(x)}\ dx-if(-\infty)\left(\alpha\overline{g(\infty)}-\overline{g(-\infty)}\right)
	=
	\int_\R f(x)\overline{h(x)}w(x)\ dx,\quad \forall f\in D_1^\alpha.
	$$
For this equality to hold in general, $g$ must satisfy $\overline\alpha{g(\infty)}=g(-\infty)$, which becomes ${g(\infty)}=\alpha g(-\infty)$ by multiplying by $\alpha$ and recalling that $|\alpha|=1$. Hence we conclude that $g\in D_1^\alpha$, and therefore $T_w$ is essentially self-adjoint on $D_1^\alpha$.

To determine the spectrum we look for solutions of $H_w^\alpha f=\lambda f$. Such solutions have the form
	$$
	f(x)
	=
	Ce^{i\lambda\int_0^xw(t)dt}.
	$$
The condition $f(\infty)=\alpha f(-\infty)$ becomes (using the relation $\alpha=e^{i\beta}$)
	$$
	\lambda\int_0^\infty w(t)\ dt
	=
	\beta+\lambda\int_0^{-\infty}w(t)\ dt+2\pi k,\quad k\in\Z
	$$
so that we conclude
	\be\label{eq:1d-weighted-eigenvalues}
	\lambda_k^\beta
	=
	\|w\|_{L^1(\R)}^{-1}(\beta+2\pi k),\quad k\in\Z.
	\en

The fact that there are no additional points in the spectrum is due to $H_w^\alpha$ having compact resolvent. Indeed, let us show that $R^\alpha_w(z)=(H_w^\alpha-z)^{-1}$, where $z\in\C\setminus\Sigma(H_w^\alpha)$, is a compact operator $ L^2_w(\R)\to D_1^\alpha\subset L^2_w(\R)$. It suffices to show that the embedding $ D_1^\alpha\subset L^2_w(\R)$ is compact. Let $K\subset D_1^\alpha$ be a bounded set. All elements of $K$ are uniformly bounded at $\pm\infty$, and therefore for every $\epsilon>0$ there exists $M>0$ such that $\int_{|x|>M}|f(x)|^2w(x)\ dx<\epsilon$ for every $f\in K$. Concluding that $K$ is compact in $ L^2_w(\R)$ is standard, using Rellich's theorem on $|x|<M$ and the smallness of the tails on $|x|>M$.
\end{proof}

\begin{remark}
We note that if the weight $w$ has the uniform bounds $0<c<w(x)<C<\infty$ for a.e $x$ then a simple change of coordinates with a uniformly bounded Jacobian transforms $w$ to a constant weight, so that the spectrum is continuous on $\R$. However, if $w\notin L^1(\R)$ does not have such uniform bounds, then not only is it not clear what the spectrum is, it is not even clear how to define a domain for $T_w$ so that it is self-adjoint.
\end{remark}

\begin{remark}\label{rek:weighted-vs-periodic}
An important observation is that the weighted case (with an $L^1$ weight) with $\alpha=1$ is completely analogous to the case of a flow on the circle. However, due to the possibility of choosing $\alpha\neq1$, the weighted case is richer than the periodic case. This is demonstrated in Theorem \ref{thm:sing-cont}.
\end{remark}

\subsection{The multi-dimensional case}
In approaching the question of self-adjointness of transport operators on weighted spaces in higher dimensions there are two main routes. One may analyze these operators directly, attempting to find appropriate domains of definition so that essential self-adjointness ensues. Alternatively, one may view the $d$-dimensional case as being made up of a family of one-dimensional fibers. We shall pursue the latter alternative, following the ideas set by von Neumann \cite{VonNeumann1949}. We use the notation of Section \ref{sec:dir-int}.

\begin{definition}\label{def:confined}
We say that a (positive) bounded weight function $0<w\in L^\infty(\R^d)$ is \emph{confined} if there exists a set $S\subset\R^{d-1}$ such that
	\begin{enumerate}
	\item
	$w(x_1,x')=1$ for all $x'\in\R^{d-1}\setminus S$, and
	\item
	$\|w(\cdot,x')\|_{L^1(\R)}<+\infty$ for all $x'\in S$.
	\end{enumerate}
The set $S$ is called the \emph{confinement region}. If there exists $M>0$ such that $\|w(\cdot,x')\|_{L^1(\R)}<M$ for all $x'\in S$ we say that $w$ is \emph{$M$-confined}.
\end{definition}

Note that a-priori there is no restriction on the size of $S$: it may have zero or full measure in $\R^{d-1}$.

\begin{theorem}\label{thm:self-adjoint-weighted}
Let $\psi:\R^{d-1}\to\R$ be a measurable and locally bounded function of $x'=(x_2,\dots,x_d)$ and let $w$ be a confined weight. Fix a measurable function $\alpha:\R^{d-1}\to\C$ with $|\alpha|\equiv1$. Define the family of essentially self-adjoint operators
	$$
	H_w^{x'}=-i\frac{\psi(x')}{w(\cdot,x')}\frac{d}{dx}: D^{\alpha(x')}\subset L^2_{w(\cdot,x')}(\R) \to  L^2_{w(\cdot,x')}(\R),\quad x'\in\R^{d-1},
	$$
where
	$$
	D^{\alpha(x')}=\left\{g\in L^2_{w(\cdot,x')}(\R)\ \Big| \ H_w^{x'}g\in L^2_{w(\cdot,x')}(\R),\ \lim_{x\to\infty}g(x)=\alpha(x')\lim_{x\to-\infty}g(x)\right\}.
	$$
We keep the same notation for their unique self-adjoint extension. Then the shear flow operator $H_w=-i\frac{\psi}{w}\frac{\p}{\p x_1}: L^2_w(\R^d)\to L^2_w(\R^d)$ may be represented as the fibered direct sum
	$$
	H_w=\int_{\R^{d-1}}^\oplus H_w^{x'}\ dx'
	$$
and is self-adjoint with domain
	\be\label{eq:weighted-domain}
	D_w
	=
	\left\{
	f\in L^2_w(\R^d)\ \Big|\ f(\cdot,x')\in D^{\alpha(x')} \text{ a.e. }x'\in\R^{d-1}, \ \int_{\R^{d-1}}\|H_w^{x'}f(\cdot,x')\|_{L^2_{w(\cdot,x')}(\R)}^2\ dx'<\infty
	\right\}.
	\en
Moreover, the spectrum $\Sigma(H_w)$ is characterized as follows: 
	\be\label{eq:weighted-eigenvalues}
	\lambda\in\Sigma(H_w)\quad\Leftrightarrow\quad\forall\epsilon>0,\ \mathcal{L}^{d-1}\left\{x'\ \big|\ \Sigma(H_w^{x'})\cap(\lambda-\epsilon,\lambda+\epsilon)\neq\emptyset\right\}>0
	\en
where $\mathcal L^{d-1}$ denotes the $d-1$ dimensional Lebesgue measure.
\end{theorem}
We refer to \cite[Theorem XIII.85]{Reed1978} for an essentially similar statement, including proof. Define $\beta:\R^{d-1}\to[0,2\pi)$ by the relation $\alpha(x')=e^{i\beta(x')}$. As we have proved in Theorem \ref{thm:weighted-1d-sa}, for any $x'\in\R^{d-1}$ for which $w(\cdot,x')$ is integrable, the operator $H_w^{x'}$ has pure point spectrum, and its eigenvalues are
	\begin{equation}\label{eq:eig-fiber}
	\lambda_k^{\beta(x')}
	=
	\|w(\cdot,x')\|_{L^1(\R)}^{-1}(\beta(x')+2\pi k),\quad k\in\Z.
	\end{equation}
Recalling the decomposition \eqref{eq:decomp-s} of $H_w$ into $H_w^S\oplus H_w^{S^c}$, we have:

\begin{corollary}\label{cor:spectrum-confined}
Under the assumptions of Theorem \ref{thm:self-adjoint-weighted} and assuming in addition that $\alpha(x')=\alpha=e^{i\beta}$ is constant, that the confinement region $S\subset\R^{d-1}$ is open and connected, and that $\|w(\cdot,x')\|_{L^1(\R)}$ is continuous in $S$ and satisfies
		\begin{equation*}
		\inf_{x'\in S}\|w(\cdot,x')\|_{L^1(\R)}=m,\qquad
		\sup_{x'\in S}\|w(\cdot,x')\|_{L^1(\R)}=M,
		\end{equation*}
then:
	\begin{enumerate}
	\item
	If $0<m=M<\infty$ then the spectrum of $H_w^S$ consists solely of a discrete part:
		\begin{equation*}
		\Sigma(H_w^S)=\frac{\beta}{m}+\frac{2\pi}{m}\Z.
		\end{equation*}
	\item
	If $0<m<M<+\infty$ then the spectrum of $H_w^S$ is:
		\begin{equation*}
		\Sigma(H_w^S)=\left(\bigcup_{-k\in\N}\left[\frac{\beta+2\pi k}{m},\frac{\beta+2\pi k}{M}\right]\right)
		\cup
		\left(\bigcup_{k\in\N\cup\{0\}}\left[\frac{\beta+2\pi k}{M},\frac{\beta+2\pi k}{m}\right]\right)
		\end{equation*}
	\item
	If $0=m<M<+\infty$ then the spectrum of $H_w^S$ is:
		\begin{equation*}
		\Sigma(H_w^S)=
		\begin{cases}\left(-\infty,-\frac{2\pi}{M}\right]
		\cup
		\{0\}
		\cup
		\left[\frac{2\pi}{M},+\infty\right)&\beta=0,\\
		\left(-\infty,\frac{\beta-2\pi}{M}\right]
		\cup
		\left[\frac{\beta}{M},+\infty\right)&\beta\neq0.
		\end{cases}
		\end{equation*}
	\item
	If $0\leq m<M=\infty$ then the spectrum of $H_w^S$ is the entire real line.
	\end{enumerate}
\end{corollary}

This simple result is completely analogous to the case of a periodic incompressible flow, with minimal and maximal periods $T_{min}$ and $T_{max}$ corresponding to $m$ and $M$. We refer to \cite{Cox2013a} for a recent result concerning such flows, including analogous proof.

\begin{remark}
There is no guarantee that the spectra appearing in parts (2), (3) and (4) of Corollary \ref{cor:spectrum-confined} are absolutely continuous. Indeed, if for example $S\ni x'\mapsto \|w(\cdot,x')\|_{L^1(\R)}\in\R$ is constant on some subset $S'\subset S$ of positive measure, there will be an embedded eigenvalue.
\end{remark}

As mentioned in Remark \ref{rek:weighted-vs-periodic}, the presence of the parameter $\alpha(x')$ provides more flexibility that does not exist in the periodic case. Let us demonstrate this:

\begin{theorem}\label{thm:sing-cont}
There exists a choice of $\alpha(x')=e^{i\beta(x')}$ for which the unique self-adjoint extension of the operator $H_w:D^{\alpha(x')}\subset L^2_w(\R^d)\to L^2_w(\R^d)$ has a spectrum which is purely singular continuous.
\end{theorem}

\begin{proof}
Let the weight $w=w(x_1)$ depend only on $x_1$, and assume that $\int|w(x_1)|\ dx_1=W<+\infty$. Then in this case $S=\R^{d-1}$ and equations \eqref{eq:weighted-eigenvalues} and \eqref{eq:eig-fiber} combined provide a full characterization of the spectrum $\Sigma(H_w)$. Let $\Ccal$ be some Cantor set contained in $(0,\frac{\pi}{W})$, and without loss of generality assume that $\frac{\pi}{W}\leq2\pi$. Then any choice of $\beta(x')$ such that $\operatorname{Ran}(\beta)=\Ccal$ produces such an operator. The spectrum in this case is purely singular continuous:
	\begin{equation*}
	\Sigma(H_w)
	=
	\Ccal+\frac{2\pi}{W}\Z.
	\end{equation*}
\end{proof}

\section{The density of states}\label{sec:parallel}


We exploit the explicit and simple unitary relationship between differential operators and multiplication operators provided by the Fourier transform, to estimate the density of states of shear flows.

\subsection{The $L^2$ case}\label{sec:parallel-unweighted}

To set ideas, we start with the simple operator
	$$
	T_{1}=-i\frac{\p}{\p x_1}:C_0^\infty(\R^d)\subset L^2(\R^d)\to L^2(\R^d).
	$$
We continue to denote its unique self-adjoint extension by $T_1$. Its domain is $D(T_1)=\int_{\R^{d-1}}^\oplus H^1(\R)$. We refer to \cite[V-\S3.3, Example 3.14]{Kato1995} for further discussion. Denote its spectral family by $\{E(\l)\}_{\lambda\in\R}$ and let
	\be\label{eq:fourier}
	\Fcal_1[r](\xi_1,x_2,\dots,x_d):=\frac{1}{\sqrt{2\pi}}\int_{-\infty}^\infty r(x_1,x_2,\dots,x_d)e^{-ix_1\xi_1}dx_1
	\en
be the partial Fourier transform with respect to the first variable.
It is well known that $\mathcal F_1:L^2(\R^d)\to L^2(\R^d)$ is a unitary operator relating $\mathcal F_1 T_1\mathcal F_1^{-1}=m_1$ where $m_1$ is the multiplication operator
	$$m_1(\xi_1,x_2,\dots,x_d)=\xi_1$$
acting in $L^2_{\xi_1,x_2,\dots,x_d}(\R^d)$.
Their spectral families are therefore related by the same unitary equivalence. This fact allows us to express $E(\l)$ using the simple expression given in \eqref{eq:spec-meas} for multiplication operators:
	$$
	\|E(\l)f\|^2_{L^2(\R^d)}=\int_{\R^{d-1}}\int_{\xi_1\leq\lambda}|\mathcal F_1[f](\xi_1,x_2,\dots,x_d)|^2\;d\xi_1 dx_2\cdots dx_d
	$$
where $f\in D(T_1)$.
Equivalently, we may write in bilinear form: given $f,g\in D(T_1)$ the spectral measure satisfies
	\be\label{eq:flat-meas}
	\left(E(\l)f,g\right)_{L^2(\R^d)}
	=
	\int_{\R^{d-1}}\int_{\xi_1\leq\lambda}\Fcal_1[f](\xi_1,x_2,\dots,x_d)\overline{\Fcal_1[g](\xi_1,x_2,\dots,x_d)}\;d\xi_1dx_2\cdots dx_d.
	\en
Therefore, whenever the mapping $\lambda\mapsto\left(E(\l)f,g\right)_{L^2(\R^d)}$ is differentiable, the density of states is given by
	\be\label{eq:flat-dens}
	\frac{d}{d\lambda}\Big|_{\l=\l_0}\left(E(\l)f,g\right)_{L^2(\R^d)}
	=
	\int_{\R^{d-1}}\Fcal_1[f](\lambda_0,x_2,\dots,x_d)\overline{\Fcal_1[g](\lambda_0,x_2,\dots,x_d)}\;dx_2\cdots dx_d.
	\en
In order to make sense of the last equation, we use the trace theorem for half-spaces \cite[Theorem 9.4]{Lions1972}. It is sufficient to assume that $\Fcal_1[f],\,\Fcal_1[g]\in H^{\sigma}(\R;L^2(\R^{d-1})),$ the Sobolev space of order $\sigma>\frac12,$ valued in $L^2(\R^{d-1}).$ In turn, by definition, this means that $f,g\in L^{2,\sigma}(\R;L^2(\R^{d-1}))$, where
	\be\label{eq:l2sigma}
	L^{2,\sigma}(\R;L^2(\R^{d-1}))
	=
	\left\{r:\R^d\to\C\quad\Big|\quad\int_{\R^d}(1+x_1^2)^{\sigma}|r(x)|^2dx<\infty\right\}.
	\en
That is, in view of the discussion of the density of states in Subsection \ref{sec:dens-states}, $L^{2,\sigma}\subset L^2(\R^d)$ is the subspace on which there is an explicit estimate for the density of states of $T_1$. The absolutely continuous subspace is, in fact, $L^2(\R^d)$ itself (see Proposition \ref{prop:abs-cont} below for a precise and more general statement). For brevity we shall denote
	$$
	\mathcal X^\sigma:=L^{2,\sigma}(\R;L^2(\R^{d-1}))
	$$
for the remainder of this section. Now we turn our attention to shear flows, as defined in equation \eqref{eq:shear-flow}.
\begin{proposition}\label{prop:dilation}
The shear flow $H_1$ is unitarily equivalent to the uniform flow $T_1$.
\end{proposition}
\begin{proof}
Consider the unitary mapping $\Ucal:L^2(\R^d)\to L^2(\R^d)$ given by
	$$
	\Ucal f(x_1,x')
	:=
	\sqrt{\psi(x')}f(\psi(x')x_1,x'),\quad(x_1,x')\in\R\times\R^{d-1}.
	$$
It is easy to verify that
	$$
	\Ucal^{-1}T_1\Ucal
	=
	H_1.
	$$
\end{proof}

It follows from Proposition \ref{prop:dilation} that the spectral families of $T_1$ and $H_1$ are unitarily equivalent by the same unitary transformation $\Ucal$, that in turn depends on $\psi$. However, in order to attain bounds for the \textit{density of states}, we need to impose further hypotheses on $\psi$:

\begin{definition}[Regular shear flow]\label{def:reg-flow}
We say that the operator $H_1$ is a \emph{regular shear flow} whenever $\psi$ satisfies the following assumptions:
	\begin{enumerate}
	\item[\textbf{A1}]
	$\psi$ is positive, bounded uniformly away from $0$: $\psi(x\rq{})>\ell>0$ for all $x\rq{}\in\R^{d-1}$,
	\item[\textbf{A2}]
	$\psi$ is globally Lipschitz with constant $L$:
		$$
		|\psi(x\rq{})-\psi(y\rq{})|\leq L|x\rq{}-y\rq{}|
		$$
	for all $x\rq{},y\rq{}\in\R^{d-1}$.
	\end{enumerate}
\end{definition}

As before, we shall obtain information on the spectral measure of $H_1$ by first considering a multiplication operator that is unitarily equivalent to it. The partial Fourier transform $\Fcal_1$ defined in \eqref{eq:fourier} defines a unitary transformation $\Fcal_1H_1\Fcal_1^{-1}=m_1^\psi$ where $m_1^\psi$ is the multiplication operator
	$$
	m_1^\psi(\xi_1,x_2,\dots,x_d)=\xi_1\psi(x_2,\dots,x_d)
	$$
acting in $L^2_{\xi_1,x_2,\dots,x_d}(\R^d)$. Denoting the spectral family of $H_1$ by $\{E(\l)\}_{\lambda\in\R}$ we obtain the following bilinear form:
	\be\label{eq:layer-meas}
	\left(E(\l)f,g\right)_{L^2(\R^d)}
	=
	\int_{\psi\xi_1\leq\lambda}\Fcal_1[f](\xi_1,x_2,\dots,x_d)\overline{\Fcal_1[g](\xi_1,x_2,\dots,x_d)}\;d\xi_1dx_2\cdots dx_d.
	\en
Then, as before, on the absolutely continuous subspace (which we identify in Proposition \ref{prop:abs-cont}) we can write the expression for the density of states as
	\be\label{eq:layer-dens}
	\frac{d}{d\lambda}\Big|_{\l=\l_0}\left(E(\l)f,g\right)_{L^2(\R^d)}
	=
	\int_{\psi\xi_1=\lambda_0}\Fcal_1[f]\left(\xi_1,x_2,\dots,x_d\right)\overline{\Fcal_1[g]\left(\xi_1,x_2,\dots,x_d\right)}|\nabla(\psi\xi_1)|^{-1}d S_{\l_0},
	\en
which is a surface integral over the $(d-1)$-dimensional surface
	$$
	\Gamma_{\l_0}:=\left\{\left(\frac{\lambda_0}{\psi(x_2,\dots,x_d)},x_2,\dots,x_d\right)\right\}_{(x_2,\dots,x_d)\in\R^{d-1}}
	$$
with $dS_{\l_0}$ being the Lebesgue surface measure. The appearance of the gradient
	$$
	\nabla(\psi\xi_1)
	=
	\left(\frac{\p(\psi\xi_1)}{\p{\xi_1}},\frac{\p(\psi\xi_1)}{\p{x_2}},\cdots,\frac{\p(\psi\xi_1)}{\p x_d}\right)
	=
	\left(\psi,\xi_1\psi_{x_2},\cdots,\xi_1\psi_{x_d}\right)
	$$
in \eqref{eq:layer-dens} is due to the coarea formula (see \cite[Appendix C.3]{Evans2010}). Since $\psi$ is assumed to be uniformly bounded away from $0$, the term $|\nabla(\psi\xi_1)|^{-1}$ is uniformly bounded. Moreover, $\Gamma_{\l_0}$ is globally Lipschitz continuous since it is the graph of $x\rq{}\mapsto\lambda_0/\psi(x\rq{})$. Indeed, we have that
	$$
	\left|\frac{1}{\psi(x\rq{})}-\frac{1}{\psi(y\rq{})}\right|
	=
	\left|\frac{\psi(y\rq{})-\psi(x\rq{})}{\psi(x\rq{})\psi(y\rq{})}\right|
	\leq
	\frac{L|y\rq{}-x\rq{}|}{\ell^2}.
	$$
We denote by
	$$
	L_{\Gamma_{\l_0}}:=\frac{|\l_0|L}{\ell^2}
	$$
the Lipschitz constant of the surface $\Gamma_{\l_0}$. As before, to make sense of the right hand side of \eqref{eq:layer-dens} we need a theorem that allows us to evaluate the traces of $\Fcal_1[f]$ and $\Fcal_1[g]$ on the hypersurface $\Gamma_{\l_0}\subseteq\R^d$. Since $\Gamma_{\l_0}$ is the graph of a Lipschitz function of the variable $x'=(x_2,\dots,x_d)$, we can derive an estimate by a straightforward computation:

\begin{theorem}\label{thm:shear-flow-density}
Let $H_1$ be a regular shear flow and let $\sigma>\frac12$. The density of states of $H_1$ is estimated by
	$$
	\left|\frac{d}{d\lambda}\Big|_{\l=\l_0}\left(E(\l)f,g\right)_{L^2(\R^d)}\right|
	\leq
	C(\sigma,\psi,\lambda_0)\|f\|_{\mathcal X^\sigma}\|g\|_{\mathcal X^\sigma}
	$$
where the subspace $\mathcal X^\sigma=L^{2,\sigma}(\R;L^2(\R^{d-1}))$ is defined in \eqref{eq:l2sigma} and where $C(\sigma,\psi,\lambda_0)=C(\sigma,\ell)(1+L_{\Gamma_{\l_0}})>0$ is a constant depending on $\sigma, \ell, L$ and $\l_0$, but not on $f$ or $g$.
\end{theorem}
\begin{proof}
First we estimate the expression \eqref{eq:fourier} for the partial Fourier transform with respect to the first variable as (here $\xi_1\in\R$)
	\be\label{eq:four-estimate}
	\left|\Fcal_1[r](\xi_1,x')\right|^2
	\leq
	C(\sigma)\int_{-\infty}^\infty (1+x_1^2)^\sigma|r(x_1,x')|^2dx_1.
	\en
We write the surface measure on $\Gamma_{\l_0}$ as
	$$
	dS_{\l_0}=a_{\l_0}(x')dx'
	$$
with
	$$
	|a_{\l_0}(x')|
	\leq
	C(1+L_{\Gamma_{\l_0}}).
	$$
Hence, integrating \eqref{eq:four-estimate} on $\Gamma_{\l_0}$ we estimate \eqref{eq:layer-dens} as follows:
	\be\label{eq:dens-states-shear}
	\begin{split}
	\left|\frac{d}{d\lambda}\Big|_{\l=\l_0}\left(E(\l)f,g\right)_{L^2(\R^d)}\right|
	&=
	\left|\int_{\Gamma_{\lambda_0}}\Fcal_1[f]\left(\xi_1,x_2,\dots,x_d\right)\overline{\Fcal_1[g]\left(\xi_1,x_2,\dots,x_d\right)}|\nabla(\psi\xi_1)|^{-1}d S_{\l_0}\right|\\
	&\leq
	\left(\sup_{(\xi_1,x\rq{})\in\Gamma_{\l_0}}|\nabla(\psi(x\rq{})\xi_1)|^{-1}\right)\|\Fcal_1[f]\|_{L^2(\Gamma_{\l_0})}\|\Fcal_1[g]\|_{L^2(\Gamma_{\l_0})}\\
	&\leq
	C(\sigma,\ell)\left(\int_{\R^d} (1+x_1^2)^\sigma|f|^2a_{\l_0} dx\right)^{\frac12}\left(\int_{\R^d} (1+x_1^2)^\sigma|g|^2a_{\l_0} dx\right)^{\frac12}\\
	&\leq
	C(\sigma,\ell)(1+L_{\Gamma_{\l_0}})\|f\|_{\mathcal X^\sigma}\|g\|_{\mathcal X^\sigma}
	\end{split}
	\en
which proves the claim.
\end{proof}

We can now give a concrete example of the abstract formula \eqref{eq:A-lambda}:
\begin{corollary}\label{cor:a-lambda}
Let $\sigma>\frac12$ and denote by $(\mathcal X^\sigma)^*=L^{2,-\sigma}(\R;L^2(\R^{d-1}))=\mathcal X^{-\sigma}$ the dual space (with respect to the $L^2$ inner product) to $\mathcal X^{\sigma}$. There exists an operator $A(\l): \mathcal X^\sigma\to \mathcal X^{-\sigma}$ satisfying
	$$\left<A(\l)f,g\right>=\frac{d}{d\mu}\Big|_{\mu=\l}(E(\mu)f,g)_{L^2(\R^d)},$$
where $\left<\cdot,\cdot\right>$ denotes the pairing of the dual spaces $(\mathcal X^{-\sigma},\mathcal X^{\sigma})$. Moreover its operator norm satisfies the bound
	\be\label{eq:Al-estimate}
	\|A(\l)\|_{\mathcal B(\mathcal X^{\sigma},\mathcal X^{-\sigma})}\leq C(\sigma,\ell)(1+L_{\Gamma_\l})
	\en
where $C(\sigma,\ell)$ is the same constant as in \eqref{eq:dens-states-shear}.
\end{corollary}

\begin{proof}
The existence of $A(\l)$ is standard, due to the fact that $\mathcal X^{-\sigma}$ is the dual space to $\mathcal X^{\sigma}$ with respect to the $L^2(\R^d)$ scalar product. The bound on its operator norm is due to the bound on the norm of the bilinear form $\frac{d}{d\mu}\big|_{\mu=\l}\left(E(\mu)\cdot,\cdot\right)_{L^2(\R^d)}$.
\end{proof}


\begin{proposition}\label{prop:abs-cont}
The spectral measure of the self-adjoint operator $H_1:D(H_1)\subset L^2(\R^d)\to L^2(\R^d)$ is absolutely continuous with respect to the Lebesgue measure.
\end{proposition}

\begin{proof}
This is a simple consequence of the bound \eqref{eq:dens-states-shear}: the signed measure $\left(E(\l)f,f\right)_{L^2(\R^d)}$ is absolutely continuous with respect to the Lebesgue measure on $\mathcal X^\sigma$ for any $\sigma>\frac12$. Since $\mathcal X^{\sigma}$ is dense in $L^2(\R^d)$ the assertion follows from the fact that the absolutely continuous space is closed \cite[Chapter 10, \S1.2]{Kato1995}.
\end{proof}


\subsection{The weighted-$L^2$ case}\label{sec:parallel-weighted}
In this section we extend our foregoing results to certain weighted cases. Theorem \ref{thm:self-adjoint-weighted} demonstrates how sensitive the spectrum of $H_w$ is to the choice of weight. While plenty can be said on this topic, we shall focus on $M$-confined weights (see Definition \ref{def:confined}).

It is natural to restrict our attention to functions that have the same limits at $\pm\infty$, in particular since this class includes the constant functions. Therefore, in view of the characterization of the domain of weighted transport operators in Theorem \ref{thm:self-adjoint-weighted}, we make the following assumption:

\begin{assumption}
We say that the domain $D_w$ (defined in \eqref{eq:weighted-domain}) is \emph{symmetric} and denote it by $D^{\text{symm}}_w$ if $\alpha(x')=1$ for every $x'\in\R^{d-1}$.
\end{assumption}

Due to the fibered structure of transport operators, the spectrum of the operator is the union of the contributions from the confinement region and its complement. Since the main reason behind the introduction of a weight is to include functions that do not decay, we introduce a new functional space, analogous to the space $\mathcal X^\sigma$ defined in \eqref{eq:l2sigma}:
	\be\label{eq:y-sigma}
	\mathcal Y^\sigma
	:=
	\left\{
	r:\R^d\to\C\ | \ 
	r(\cdot,x')\in L^{2,\sigma}(\R),\ x'\notin S,\quad
	r(\cdot,x')\in L^2_{w(\cdot,x')}(\R),\ x'\in S
	\right\}.
	\en

As we shall see, this space is the natural space to consider. The norm on this space is defined as
	$$
	\|r\|_{\mathcal Y^\sigma}^2
	=
	\int_S\|r(\cdot,x')\|_{L^2_{w(\cdot,x')}(\R)}^2\ dx'
	+
	\int_{\R^{d-1}\setminus S}\|r(\cdot,x')\|_{L^{2,\sigma}(\R)}^2\ dx'.
	$$
We have the following simple characterization of the spectrum.

\begin{theorem}\label{thm:dens-states-weighted-parallel}
Let $\psi:\R^{d-1}\to\R$ satisfy the assumptions of Definition \ref{def:reg-flow}, let $w$ be an $M$-confined weight function on $\R^{d}$ with confinement region $S\subset\R^{d-1}$ and let $\sigma>\frac12$. Assume that $\R^{d-1}\setminus S$ has positive measure. Then there exists $\delta>0$ such that the self-adjoint operator $H_w=-i\frac{\psi}{w}\frac{\p}{\p x_1}:D^{\text{symm}}_w\subset L^2_{w}(\R^d)\to L^2_{w}(\R^d)$ has an absolutely continuous spectrum in $(-\delta,\delta)\setminus\{0\}$ (with no embedded eigenvalues). Moreover, the density of states satisfies
	\be\label{eq:est-dens-states-confined}
	\left|\frac{d}{d\lambda}\Big|_{\l=\l_0}\left(E_w(\l)f,g\right)_{L^2_{w}(\R^d)}\right|
	\leq
	C(\sigma,\ell,L,\lambda_0)\|f\|_{\mathcal Y^\sigma}\|g\|_{\mathcal Y^\sigma},
	\quad 0<|\lambda_0|<\delta,
	\en
where $C(\sigma,\ell,L,\lambda_0)$ is a constant depending only on its arguments but not on $f$ or $g$ (here $\ell$ and $L$ are parameters related to $\psi$ as in Definition \ref{def:reg-flow}).
\end{theorem}

\begin{remark}\label{rek:constant}
The constant $C(\sigma,\ell,L,\lambda_0)$ may be expressed more explicitly as
	$$
	C(\sigma,\ell,L,\lambda_0)=C(\sigma,\ell)(1+L_{\Gamma_{\lambda_0}})
	$$
where $L_{\Gamma_{\lambda_0}}=\frac{|\lambda_0|L}{\ell^2}$ is the Lipschitz constant of the surface $\Gamma_{\lambda_0}$. We refer to Theorem \ref{thm:shear-flow-density} and the discussion preceding it for further detail.
\end{remark}

\begin{proof}
The results of subsection \ref{sec:dir-int} imply that the spectrum of $H_w$ decomposes into contributions from $S$ and from its complement $S^c$. The spectrum of $H_w$ restricted to $S^c$ is simply $\R$ and due to Corollary \ref{cor:spectrum-confined} the spectrum of the restriction to $S$ has a spectral gap depending on $M$, except for a possible eigenvalue (which may be of infinite multiplicity) at zero (for this only the bound $M$ is important, rather than the assumptions on $S$ or the continuity of the $L^1$ norm of the fibers of the weight functions). Therefore there indeed exists such a $\delta=\delta(M)>0$.

We now recall the fiber decomposition \eqref{eq:direct-int-spec-family} of the spectral family $E_w(\lambda)$ of $H_w$
	$$
	E_w(\lambda)
	=
	\int_{\R^{d-1}}^\oplus E_w^{x'}(\lambda)\ dx'
	$$
where $\{E_w^{x'}(\lambda)\}_{\lambda\in\R}$ is the spectral family of the operator $H_w^{x'}$. Hence the expression for the density of states becomes
	$$
	\frac{d}{d\lambda}\Big|_{\l=\l_0}\left(E_w(\l)f,g\right)_{L^2_w(\R^d)}
	=
	\frac{d}{d\lambda}\Big|_{\l=\l_0}\int_{\R^{d-1}}\left(E_w^{x'}(\l)f(\cdot,x'),g(\cdot,x')\right)_{L^2_{w(\cdot,x')}(\R^d)}\ dx'.
	$$
However for $0<|\lambda_0|<\delta$ the fibers in $S$ do not contribute to the density of states, and, on the other hand, the weight in $\R^{d-1}\setminus S$ is identically $1$ so that the space $ L^2_{w(\cdot,x')}(\R)$ is simply $L^2(\R)$. Hence, if $S\neq\R^{d-1}$, we have
	\begin{align*}
	\frac{d}{d\lambda}\Big|_{\l=\l_0}\left(E_w(\l)f,g\right)_{ L^2_{w}(\R^d)}
	&=
	\frac{d}{d\lambda}\Big|_{\l=\l_0}\int_{\R^{d-1}\setminus S}\left(E_w^{x'}(\l)f(\cdot,x'),g(\cdot,x')\right)_{ L^2_{w(\cdot,x')}(\R)}\ dx'\\
	&=
	\frac{d}{d\lambda}\Big|_{\l=\l_0}\int_{\R^{d-1}\setminus S}\left(E_w^{x'}(\l)f(\cdot,x'),g(\cdot,x')\right)_{L^2(\R)}\ dx'\\
	&=
	\frac{d}{d\lambda}\Big|_{\l=\l_0}\int_{\R^{d-1}\setminus S}\int_{\psi(x')\xi_1\leq\lambda}\mathcal F_1[f](\xi_1,x')\overline{\mathcal F[g](\xi_1,x')}\ d\xi_1\ dx'\\
	&=
	\int_{x'\notin S,\ \psi\xi_1=\lambda_0}\mathcal F_1[f](\xi_1,x')\overline{\mathcal F[g](\xi_1,x')}|\nabla(\psi\xi_1)|^{-1}\ dS_{\lambda_0}.
	\end{align*}
Recalling estimate \eqref{eq:dens-states-shear} we may estimate
	\be
	\left|\frac{d}{d\lambda}\Big|_{\l=\l_0}\left(E_w(\l)f,g\right)_{ L^2_{w}(\R^d)}\right|
	\leq
	C(\sigma,\ell,L,\lambda_0)\|f\|_{\mathcal Y^\sigma}\|g\|_{\mathcal Y^\sigma}.
	\en
If $S=\R^{d-1}$ the density of states is simply $0$.
\end{proof}

\section{A uniform ergodic theorem}\label{sec:ergodic-parallel}
John von Neumann's classic ergodic theorem is:
\begin{theorem}[\cite{VonNeumann1932a}]
Let $G_t$ be a one-parameter group of measure-preserving transformations of a measure space $(\Omega,\mu)$. Let $P$ be the orthogonal projection onto $\{v\in L^2(\Omega,d\mu)\ |\ \forall t,\ v\circ G_t=v\}$. Then for any $f\in L^2(\Omega,d\mu)$
	$$
	\lim_{T\to\infty}\frac 1{2T}\int_{-T}^Tf\circ G_t\ dt=Pf.
	$$
\end{theorem}
This is a statement on \emph{strong} convergence. Another well known result is Birkhoff's ergodic theorem \cite{Birkhoff1931} which deals with \emph{pointwise} convergence. A good reference for both theorems is \cite{Reed1981}. Using Theroem \ref{thm:dens-states-weighted-parallel} we now show \emph{uniform} convergence (or convergence \emph{in operator norm}) on a certain subspace:

\begin{theorem}
Let $\psi:\R^{d-1}\to\R$ satisfy the assumptions of Definition \ref{def:reg-flow}, let $w$ be an $M$-confined weight function on $\R^{d}$ with confinement region $S\subset\R^{d-1}$ and let $\sigma>\frac12$. Assume that $\R^{d-1}\setminus S$ has positive measure. Consider the self-adjoint operator $H_w=-i\frac{\psi}{w}\frac{\p}{\p x_1}:D^{\text{symm}}_w\subset L^2_{w}(\R^d)\to L^2_{w}(\R^d)$. Then
	$$
	\lim_{T\to\infty}\frac 1{2T}\int_{-T}^T e^{itH_w}\ dt=P_w
	$$
in the uniform operator topology on $\mathcal B(\mathcal Y^\sigma,\mathcal Y^{-\sigma})$, where $P_w$ is the orthogonal projection onto the kernel of $H_w$. Here $\mathcal Y^\sigma$ is as defined in \eqref{eq:y-sigma}.
\end{theorem}

\begin{proof}
For brevity, define the operator
	$$
	P^T_w=\frac 1{2T}\int_{-T}^T e^{itH_w} dt\in\mathcal B( L^2_{w}(\R^d), L^2_{w}(\R^d)).
	$$
In terms of the spectral family $\{E_w(\lambda)\}_{\lambda\in\R}$ of $H_w$,
	\be
	P^T_wf=\frac{1}{2T}\int_{-T}^T\int_\R e^{it\l}dE_w(\l)f\;dt,
	\en
so that, assuming sufficient regularity on $f$,
	\be
	P^T_wf=\int_\R\frac{\sin{T\l}}{T\l}dE_w(\l)f.
	\en
On the other hand, the projection operator $P_w$ is expressed as
	\be
	P_wf=\int_\R\chi(\l)dE_w(\l)f
	\en 
where $\chi(0)=1$ and $\chi(\l)=0$ whenever $\l\neq0$. Therefore the difference of the two operators is
	\be
	\left(P^T_w-P_w\right)f
	=
	\int_{\R\setminus\{0\}}\frac{\sin{T\l}}{T\l}dE_w(\l)f.
	\en
Our strategy is to break up the domain of integration as
	$$
	\int_{\R\setminus\{0\}}=\int_{(-\infty,-\epsilon]}+\int_{(-\epsilon,0)}+\int_{(0,\epsilon)}+\int_{[\epsilon,\infty)}=I_1+I_2+I_3+I_4
	$$
(where $\epsilon>0$) and estimate each term separately. We shall focus on $I_3$ and $I_4$; the integrals $I_1$ and $I_2$ are treated in an identical fashion. Consider first the term $I_3$. Recall that in Theorem \ref{thm:dens-states-weighted-parallel} it was shown that there exists $\delta>0$ such that the spectral measure of $H_w$ is absolutely continuous in $(-\delta,\delta)\setminus\{0\}$. Therefore if $\epsilon<\delta$ the estimate \eqref{eq:est-dens-states-confined} holds. Hence we can replace $dE_w(\l)$ by $A_w(\l) d\l$ to get
	\begin{align*}
	\left\|I_3\right\|_{\mathcal Y^{-\sigma}}^2
	&=
	\left\|\int_{(0,\epsilon)}\frac{\sin{T\l}}{T\lambda}A_w(\l)f\ d\l\right\|_{\mathcal Y^{-\sigma}}^2\\
	&\leq
	C(\sigma,\ell,L,\epsilon)\|f\|_{\mathcal Y^\sigma}^2\int_{(0,\epsilon)}\left|\frac{\sin{T\l}}{T\lambda}\right|^2d\lambda.
	\end{align*}
Recalling Remark \ref{rek:constant}, the constant has the form $C(\sigma,\ell,L,\epsilon)=C(\sigma,\ell)(1+\epsilon L\ell^{-2})$, so that
	\be
	\left\|I_3\right\|_{\mathcal Y^{-\sigma}}^2
	<
	\epsilon (1+\epsilon L\ell^{-2})C(\sigma,\ell)\|f\|_{\mathcal Y^\sigma}^2.
	\en
Turning to $I_4$ we have:
	\be
	\begin{split}
	\left\|I_4\right\|_{ L^2_{w}(\R^d)}^2
	&=
	\int_{[\epsilon,\infty)}\left|\frac{\sin{T\l}}{T\lambda}\right|^2d\left\|E_w(\l)f\right\|_{ L^2_{w}(\R^d)}^2\\
	&\leq
	\frac{1}{T^2\epsilon^2}\int_{[\epsilon,\infty)}d\left\|E_w(\l)f\right\|_{ L^2_{w}(\R^d)}^2\\
	&\leq
	\frac{1}{T^2\epsilon^2}\int_{\R}d\left\|E_w(\l)f\right\|_{ L^2_{w}(\R^d)}^2\\
	&=
	\frac{1}{T^2\epsilon^2}\|f\|_{ L^2_{w}(\R^d)}^2.
	\end{split}
	\en
We therefore conclude that
	\be
	\begin{split}
	\left\|\left(P^T_w-P_w\right) f\right\|_{\mathcal Y^{-\sigma}}^2
	&<
	2\epsilon \left(1+\frac{\epsilon L}{\ell^{2}}\right)C(\sigma,\ell)\|f\|_{\mathcal Y^\sigma}^2
	+
	\frac{2}{T^2\epsilon^2}\|f\|_{ L^2_{w}(\R^d)}^2\\
	&\leq
	C'(\sigma,\ell)\left(\epsilon \left(1+\frac{\epsilon L}{\ell^{2}}\right)+\frac{1}{T^2\epsilon^2}\right)\|f\|_{\mathcal Y^\sigma}^2.
	\end{split}
	\en
Since this estimate holds for every $\epsilon\in(0,\delta)$, it follows that
	\be
	\lim\limits_{T\to\infty}\|P^T_w-P_w\|_{\mathcal B(\mathcal Y^\sigma,\mathcal Y^{-\sigma})}=0.
	\en
The rate of convergence is $T^{-2/3}$.
\end{proof}

\begin{remark}
The rate of convergence can possibly be further improved by closer inspection of the estimates in the proof.
\end{remark}

\bibliography{library}
\bibliographystyle{abbrv}
\end{document}